\documentclass[11pt]{article}

\usepackage{amsmath, amsthm, amssymb, tikz, hyperref, enumerate}
\usepackage{amsfonts}
\usepackage{fullpage}
\usepackage[enableskew,vcentermath]{youngtab}

\newcommand\beq{\begin{equation}}
\newcommand\eeq{\end{equation}}
\newcommand\bce{\begin{center}}
\newcommand\ece{\end{center}}
\newcommand\bea{\begin{eqnarray}}
\newcommand\eea{\end{eqnarray}}
\newcommand\ba{\begin{array}}
\newcommand\ea{\end{array}}
\newcommand\ben{\begin{enumerate}}
\newcommand\een{\end{enumerate}}
\newcommand\bit{\begin{itemize}}
\newcommand\eit{\end{itemize}}
\newcommand\brr{\begin{array}}
\newcommand\err{\end{array}}
\newcommand\bt{\begin{tabular}}
\newcommand\et{\end{tabular}}

\newtheorem{theorem}{Theorem}[section]
\newtheorem*{theorem*}{Theorem}

\newtheorem{cor}[theorem]{Corollary}
\newtheorem{definition}[theorem]{Definition}
\newtheorem{example}[theorem]{Example}

\newtheorem{defn}[theorem]{Definition}

\newtheorem{remark}[theorem]{Remark}

\newcommand{\todo}[1]{\vspace{5 mm}\par \noindent
\marginpar{\textsc{ToDo}} \framebox{\begin{minipage}[c]{0.95
\textwidth}
 \end{minipage}}\vspace{5 mm}\par}

\begin{document}
\title{Counting King Permutations on the Cylinder}
\author{Eli Bagno, Estrella Eisenberg, Shulamit Reches and Moriah Sigron}


\maketitle
\begin{abstract}
We call a permutation $\sigma=[\sigma_1,\dots,\sigma_n] \in S_n$ a {\em cylindrical king permutation} if $ |\sigma_i-\sigma_{i+1}|>1$ for each  $1\leq i \leq n-1$ and  $|\sigma_1-\sigma_n|>1$. 
We present some results regarding the distribution 
of the cylindrical king permutations, including some interesting recursions. We also calculate their asymptotic proportion in the set of the 'king permutations', i.e. the ones which satisfy only the first of the two conditions above. With this aim we define a new parameter on permutations, namely, the number of {\em cyclic bonds} which is a modification of the number of bonds. In addition, we present some results regarding the distribution of this parameter.\\

\end{abstract}
\section{Introduction}
In a recent paper \cite{BERS_KINGS}, the authors dealt with the set of king permutations. The inspiration was Hertzsprung's problem \cite{Sl},dealing with the number of ways to arrange $n$ non-attacking chess kings on an $n \times n$ chess board such that each column and each row contains exactly one chess king.  

There is a tight connection between the chess problem and the symmetric group $S_n$.
If we consider a permutation $\sigma=[\sigma_1,\dots,\sigma_n]\in S_n$ in a geometrical way as the set of all lattice points of the form $(i,\sigma_i)$ where $1 \leq i \leq n$, the problem of finding all the ways to arrange $n$ non-attacking chess kings is equivalent to the problem of finding all permutations $\sigma \in S_n$ such that for each $1 < i  \leq n$, $|\sigma_i-\sigma_{i-1}|>1$. This set is counted in OEIS A002464. 


Let $K_n$ be the set of all such permutations in $S_n$.  In this paper we call them simply {\it king permutations} or just {\it kings}.  
For example: 
$K_1=S_1, K_2=K_3=\emptyset$, $K_4=\{[3142],[2413]\}$. 
An explicit formula for the number of king permutations was given by Robbins \cite{Ro}. He also showed that when $n$ tends to infinity, 
the probability of picking such a permutation from $S_n$ approaches $e^{-2}$.

A natural question is to extend the Hertzsprung's problem to a celebrated variant of the chess game, namely the cylindrical chess, in which the right and left edges of the board are imagined to be joined in such a way that a piece can move off one edge and reappear at the opposite edge, moving in the same or a parallel line \cite{JELIS}.


\begin{definition}
Let $CK_n$ be defined as

$$CK_n=\{\sigma=[\sigma_1,\dots,\sigma_n] \in S_n: |\sigma_i-\sigma_{i+1}|>1,|\sigma_1-\sigma_n|>1,1\leq i \leq n-1\}.$$ 
An element of $CK_n$ will be called a {\em cylindrical king}.
\end{definition}
For example: $CK_1=S_1, CK_2=CK_3=CK_4=\emptyset$, $CK_5=\{[31425],[14253],[42531],[25314],$\\$[53142],[24135],[41352],[13524],[35241],[52413]\}$.\\
Note that each cylindrical king can also be seen as a directed Hamiltonian path in the complement of the n-cycle graph. 
The sets $CK_n$ are counted by the sequence A002493 of OEIS. The first $8$ elements are $1,0,0,0,10,60,462,3920,36954$. They were counted by Abramson and Moser in \cite{AM}.\\ 

This paper deals with some aspects of counting cylindrical king permutations. In order to facilitate the counting, we define a new concept: the {\it cyclic bond}, which extends the known concept of a bond appeared in \cite{BERS_KINGS}.\\
A {\it bond} in a permutation $\pi\in S_n$ is a length $2$ consecutive sub-sequence of adjacent numbers. Note that a king permutation is actually a permutation without bonds. This point of view can be used also to describe the set of cylindrical king permutations, provided that we slightly modify the definition of the bond in order to obtain what we call in this paper a {\em cyclic bond}. 
\begin{defn}\label{definition of blocks}
Let $\pi =[\pi_1,\ldots,\pi_n]\in S_n$ and let $i \in [n-1]$. We say that the pair $(\pi_i,\pi_{i+1})$ is a {\em (regular) bond} in $\pi$ if $\pi_{i}-\pi_{i+1} =\pm 1$. If $\pi_{n}-\pi_1=\pm 1$ then we say that the pair $(\pi_n,\pi_1)$ is an {\em edge bond} of $\pi$. In general, adopting the convention that $\pi_{n+1}=\pi_1$, we say that $(\pi_i,\pi_{i+1})$ is a cyclic bond if it is a regular or an edge bond.
\end{defn}

\begin{example}
In $\pi=[41325]$ there are 2 cyclic bonds.
The regular bond (3,2) and the edge bond (5,4).
\end{example}
According to that new definition, a permutation is a cylindrical king if and only if it has no cyclic bonds. 
This type of modification from regular to cyclic parameters has been used also in the case of the descent parameter. See for example \cite{Cellini},\cite{Petersen}.\\

Aside from its role as identifying the cylindrical kings, the definition of cyclic bonds leads to some interesting counting results by itself.

For each $\pi \in S_n$ we denote by $bnd(\pi)$ the number of regular bonds in $\pi$ and by $cbnd(\pi)$ the number of cyclic bonds in $\pi$.
\begin{definition}\label{init}
Let $B_0(t)=CB_0(t)=1$ and for $n \geq 1$: 
$$B_n(t)=\sum\limits_{\pi \in S_n}t^{bnd(\pi)},$$ and  $$CB_n(t)=\sum\limits_{\pi \in S_n}t^{cbnd(\pi)}.$$
 A simple calculations shows that:
$B_1(t)=CB_1(t)=1$,
$B_2(t)=2t$, $CB_2(t)=2t^2$,
$B_3(t)=2t^2+4t$ and $CB_3(t)=6t^2$.
Note that we chose to consider the permutations of $S_2$ as having $2$ cyclic bonds each. \\
\end{definition}
In Section \ref{cyclic bonds} we present some results regarding the counting of cyclic bonds and their effect on the structure of the cylindrical kings.\\ 
First we present relations between the number of regular bonds and the number of cyclic bonds 
(Theorem \ref{cbond1} ) for  $n \geq 2$:

  $$CB_{n+1}(t)=(n+1)B_n(t)+2(n+1)\sum _{i=1}^n(t-1)^iB_{n-i}(t),$$ where the initial conditions are as in
Definition \ref{init} .

We also have (Theorem \ref{cbond2})
for $n \geq 1$:
$$
B_n(t)=CB_n(t)+\frac{1}{n}(1-t)CB_n'(t)
$$

Then we introduce a recursion for the number of cyclic bonds. (Theorem \ref{cbond3} ) for $n \geq 2$:  
$$CB_{n+1}(t)=(n+1)[CB_n(t)+2\sum_{i=1}^{n}(t-1)^iCB_{n-i}(t)]+\frac{n+1}{n}(1-t)[CB_n'(t)+2\sum_{i=1}^n(t-1)^iCB_{n-i}'(t)]$$ where the initial conditions are as in Definition \ref{init}.

 These results have impact on the structure of the cylindrical kings and thus they enable us to introduce a recursion formula for the cylindrical kings along with a recursion that connects the number of kings and the number of cylindrical kings. (Corollary \ref{cor}):
 
  For $n \geq 2$:
\begin{enumerate}
\item $|CK_{n+1}|=(n+1)|K_n|+2(n+1)\sum _{i=1}^n(-1)^i|K_{n-i}|$ while $|K_0|=1$. 
\item $
|K_n|=|CK_n|+\frac{1}{n}|CB_{1,n}|$
\item $|CK_{n+1}|=(n+1)(|CK_n|+2\sum_{i=1}^{n}(-1)^i|CK_{n-i}|)+\frac{n+1}{n}(|CB_{1,n}|+2\sum_{i=1}^n(-1)^i|CB_{1,n-i}|)$
\end{enumerate}
where $|CB_{1,k}|$ is the number of permutations $\pi \in S_k$ with a single cyclic bond.

Section \ref{another recursion} exhibits a recursion connecting between the sizes of the sets of cylindrical
kings and non-cylindrical kings using a combinatorial proof. (Theorem \ref{recursion}):
\begin{center}
    $|A_{n}|=2|K_{n-1}|+|A_{n-2}|,$
where $A_n=K_n-CK_n.$
\end{center}

This enables us to calculate the asymptotic ratio of the set of cylindrical kings in the entire set of the kings. (Theorem \ref{asy}):
\begin{center}
   The asymptotic value of $|CK_n|/|K_n|$ is equal to 1. 
\end{center}

Finally, in Section \ref{generating} we calculate the distribution of the cylindrical king permutations using the distribution of the cyclic bonds. (Theorem \ref{TM:Generating Function for cyclic bonds} and Theorem \ref{generating of the cyclic kings})

\section{Cyclic bonds and their effect on cylindrical kings }\label{cyclic bonds}

\subsection{The number of regular bonds vs. the number of cyclic bonds}

Recall that for each $\pi \in S_n$ we denote by $bnd(\pi)$ the number of regular bonds in $\pi$ and by $cbnd(\pi)$ the number of cyclic bonds in $\pi$. Let $C_n$ be the cyclic sub-group of $S_n$
generated by the cycle $\omega=(1,2,...,n)$.
Note that the parameter $cbnd$ is invariant under the right action of $C_n$ (left shift of the one-line notation of $\pi$). Explicitly:
\begin{equation}\label{w}
cbnd(\pi)=cbnd(\pi \omega^i) \text{ for } 1 \leq i \leq n-1. 
\end{equation} 
\begin{example}
For $n=4$ one can easily see that if $\pi=[2134]$ then $\pi \omega=[1342],\,\pi \omega^2=[3421]$ and $\pi \omega^3=[4213]$. All of them have exactly $2$ cyclic bonds.  
\end{example} 
In order to count the permutations in $S_{n+1}$ with exactly $k$ cyclic bonds, we count only those permutations whose last element is $n+1$, and multiply by $n+1$.
Formally, 
\begin{equation}
\label{n+1}
 |\{\pi \in S_{n+1}:cbnd(\pi)=k\}|=(n+1)|\{\pi \in S_{n+1}:cbnd(\pi)=k\, \, and \,\, \pi_{n+1}=n+1\}|   
\end{equation}

Consider a permutation $\pi \in S_{n+1}$ such that  $\pi_{n+1}=n+1$, and define $\pi' \in S_{n}$ such that $\pi'_i=\pi_i$ for $1\leq i \leq n$. Then we have:

\begin{equation}
\label{case}
cbnd(\pi)=
\begin{cases}
 bnd(\pi')+1& \text{if } \pi'_n=n \text{  or } \pi'_1=n\\ 
bnd(\pi')  &{\text otherwise}.
\end{cases} 
\end{equation}

We discuss now the connection between the distributions of bonds and cyclic bonds (using Definition \ref{init}).

If we denote $$CZ_n(t)=\sum\limits_{\pi \in S_n,\pi_n=n}t^{cbnd(\pi)}.$$
then according to (\ref{n+1}) we have:
\begin{equation}
\label{zn}
CB_{n+1}(t)=(n+1)CZ_{n+1}(t) 
\end{equation}

Also, for $n\geq 2$ let
$$B_n^1(t)=\sum\limits_{\pi \in S_n , \pi_1=n} t^{bnd(\pi)}$$ and $$B_n^2(t)=\sum\limits_{\pi \in S_n , \pi_n=n} t^{bnd(\pi)}.$$
Then we have by (\ref{case}) for $n\geq 2:$
\begin{equation}\label{cz}
 CZ_{n+1}(t)=(B_n(t)-B_n^1(t)-B_n^2(t))+t(B_n^1(t)+B_n^2(t))   
\end{equation}

and thus according to (\ref{zn}) and (\ref{cz}):
\begin{equation}
\label{A} 
\begin{array}{ccc}
   CB_{n+1}(t)  & = & (n+1)[(B_n(t)-B_n^1(t)-B_n^2(t))+t(B_n^1(t)+B_n^2(t))]\\
    &= & (n+1)B_n(t)+(n+1)(t-1)(B_n^1(t)+B_n^2(t))
\end{array}
\end{equation}
In addition for $n\geq 2:$
$$B_n^1(t)=(B_{n-1}(t)-B_{n-1}^1(t))+tB_{n-1}^1(t)=B_{n-1}(t)+(t-1)B_{n-1}^
1(t)$$
$$B_n^2(t)=(B_{n-1}(t)-B_{n-1}^2(t))+tB_{n-1}^2(t)=B_{n-1}(t)+(t-1)B_{n-1}^2(t)$$
where $B^1_1(t)=B^2_1(t)=1$\\

Let us denote $X_n(t)=B_n^1(t)+B_n^2(t)$, thus for $n\geq 2:$

\begin{equation}
\label{B}
 X_n(t)=2B_{n-1}(t)+(t-1)X_{n-1}(t)=2B_{n-1}(t)+(t-1)(2B_{n-2}(t)+(t-1)X_{n-2}(t)).  
\end{equation}
A simple induction will show that for $n\geq 2:$
$$X_n(t)=2\sum _{i=1}^n(t-1)^{i-1}B_{n-i}(t)$$

Using (\ref{A}) and (\ref{B}) we have for $n\geq 2:$
$$CB_{n+1}(t)=(n+1)B_n(t)+(n+1)(t-1)X_n(t)=(n+1)B_n(t)+2(n+1)(t-1) \sum _{i=1}^n(t-1)^{i-1}B_{n-i}(t)$$

According to the above calculations we obtain the following theorem:
\begin{theorem}
\label{cbond1}
Let $n \geq 2$. Then  $$CB_{n+1}(t)=(n+1)B_n(t)+2(n+1)\sum _{i=1}^n(t-1)^iB_{n-i}(t)$$ where the initial conditions are as in Definition \ref{init}.
\end{theorem}

In order to enable the calculation of $CB_{n+1}(t)$ using the above recursion, we just have to be able to calculate the polynomials $B_i(t)$ for $0 \leq i \leq n$. 

The following result provides a direct formula for $B_n(t)$ using $CB_n(t)$. 

\begin{theorem}
\label{cbond2}
Let $n \geq 1$. Then we have:
$$
B_n(t)=CB_n(t)+\frac{1}{n}(1-t)CB_n'(t)
$$
\end{theorem}

\begin{proof}
Let $\pi \in S_n$ be such that $cbnd(\pi)=k$, and denote by $[\pi]$ the orbit of $\pi$ under the right action of $C_n$. The contribution of $[\pi]$ to the polynomial $CB_n(t)$ is $nt^k$, and thus its contribution to the R.H.S is $nt^k+(1-t)kt^{k-1}$. In order to complete the proof, we have to find the contribution of $[\pi]$ to $B_n(t)$. Writing $\pi=[\pi_1,\dots,\pi_{n}]$, each representative of its orbit starts with some $\pi_i$, $1\leq i \leq n$. The number of regular bonds in $[\pi_i,\pi_{i+1},\dots,\pi_1,\dots,\pi_{i-1}]$ is $k-1$ if $[\pi_{i-1},\pi_i]$ is a cyclical bond and $k$ otherwise. We conclude that the contribution of $[\pi]$ to $B_n(t)$ is exactly $kt^{k-1}+(n-k)t^k=nt^k+(1-t)kt^{k-1}$ as required.

\end{proof}
Using Theorems \ref{cbond1}, and  \ref{cbond2}, we can create a recursion for the generating function of the cyclic bonds.\\
\begin{theorem}\label{cbond3}
Let $n \geq 2$. Then 
$$CB_{n+1}(t)=(n+1)[CB_n(t)+2\sum_{i=1}^{n}(t-1)^iCB_{n-i}(t)]+\frac{n+1}{n}(1-t)[CB_n'(t)+2\sum_{i=1}^n(t-1)^iCB_{n-i}'(t)]$$
where the initial conditions are as in Definition \ref{init}.
\end{theorem}

\subsection{The number of kings vs. the number of cylindrical kings}
Recall that a king permutations is actually a permutation without bonds. As a result, we get that $|K_n|=B_n(0)$.
Moreover, a permutation is a cylindrical king if and only if it has no cyclic bonds, thus $|CK_n|=CB_n(0)$.
By Theorems \ref{cbond1}, \ref{cbond2}, and \ref{cbond3}, we have the following corollary: 

\begin{cor}
\label{cor}
For $n \geq 2$:
\begin{enumerate}
\item $$|CK_{n+1}|=(n+1)|K_n|+2(n+1)\sum _{i=1}^n(-1)^i|K_{n-i}|$$ while $|K_0|=1$. 
\item $$
|K_n|=|CK_n|+\frac{1}{n}|CB_{1,n}|$$ 
\item $$|CK_{n+1}|=(n+1)(|CK_n|+2\sum_{i=1}^{n}(-1)^i|CK_{n-i}|)+\frac{n+1}{n}(|CB_{1,n}|+2\sum_{i=1}^n(-1)^i|CB_{1,n-i}|)$$
\end{enumerate}
where $|CB_{1,k}|$ is the number of permutations $\pi \in S_k$ with a single cyclic bond.
\end{cor}

\section{Another recursion and the asymptotic value of $\frac{|CK_n|}{|K_n|}$.}\label{another recursion}
 
In this section we introduce another recursion connecting between the sizes of the sets of cylindrical kings and non-cylindrical kings and present a combinatorial proof for this recursion. Eventually, this will allow us to 
prove that the proportion of the set of cylindrical kings in the entire set of the kings is asymptotically $1$.
We start with some notations.
First, define for each $n \geq 1$ the set of kings which are not cylindrical: $$A_n=K_n-CK_n.$$ 

\begin{theorem}
\label{recursion}
We have the following recursion:

$$|A_{n}|=2|K_{n-1}|+|A_{n-2}|.$$

\end{theorem}

\begin{proof}





We present $A_n$ as a disjoint union of $4$ subsets as follows:

$$B_n^0=\{\pi\in K_n \mid \pi=[k,k+2,\dots,k-1,k+1], 2\leq k \leq n-2 \} \cup$$ $$\{\pi\in K_n \mid \pi=[k+1,k-1,\dots,k+2,k], 2\leq k \leq n-2 \}.$$

Note that by omitting the first or the last element in $\pi \in A_n^0$ we obtain a non king permutation. However, by omitting both  elements we obtain a permutation of $A_{n-2}$.  

Now, let 
$$B_n^1=\{\pi\in K_n \mid \pi=[k,l,\dots,k-1,k+1], 2\leq k \leq n-1 ,l\neq k+2\} \cup$$  $$\{\pi\in K_n \mid \pi=[k+1,k-1,\dots,l,k], 2\leq n-1 \leq n-2 , l\neq k+2\}.$$

Here, by omitting $k+1$ from $\pi \in B_n^1$ we get a permutation of $A_{n-1}$, while by omitting $k$ from $\pi$ we get a non-king of $S_{n-1}$. 

Similarly, define: 

$$B_n^2=\{\pi\in K_n \mid \pi=[k,k+2,\dots,m,k+1], 1\leq k \leq n-2 ,m\neq k-1\} \cup$$ $$\{\pi\in K_n \mid \pi=[k+1,m,\dots,k+2,k], 1\leq k \leq k-2 , m\neq k-1\},$$

and note that here by omitting $k$ from $\pi \in B_n^2$ we get a permutation of $A_{n-1}$, while by omitting $k+1$ from $\pi$ we get a non-king of $S_{n-1}$. 

Finally, define 
$$B_n^3=\{\pi\in K_n \mid \pi=[k,m,\dots,l,k+1], 1\leq k \leq n-1 ,l\neq k-1 ,m \neq k+2\}\cup$$  $$\{\pi\in K_n \mid \pi=[k+1,l,\dots,m,k]\mid  1\leq k \leq k-1 , l\neq k-1 ,m \neq k+2\}.$$

Here, by omitting any one of $k,k+1$ we get a permutation of $CK_{n-1}$. 
We prove the recursion by constructing the following functions:

\begin{enumerate}
    \item A bijection $f_0:B_n^0 \rightarrow A_{n-2}$ which implies that $|B_n^0|=|A_{n-2}|$
    \item Two bijections $f_1:B_{n}^1\rightarrow A_{n-1}$ and $f_2:B_n^2\rightarrow A_{n-1}$ which imply that $|B_n^1|=|A_{n-1}|=|B_n^2|$.
    \item A $2$ to $1$ mapping $f_3:B_n^3 \rightarrow CK_{n-1}$, which implies that $2|CK_{n-1}|=|B_n^3|$. 
\end{enumerate}

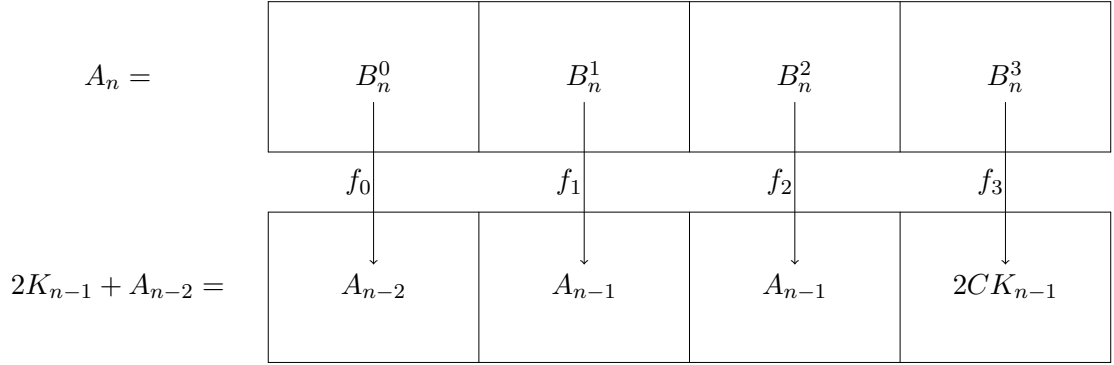
\begin{tikzpicture}[cap=round,line width=.2pt,scale=0.4pt,]
\label{A_n picture}

\draw[-](0+5,0)--(28+5,0);
\draw[-](0+5,-5)--(28+5,-5);
\draw[-](0+5,0)--(0+5,-5);
\draw[-](7+5,0)--(7+5,-5);
\draw[-](14+5,0)--(14+5,-5);
\draw[-](21+5,0)--(21+5,-5);
\draw[-](28+5,0)--(28+5,-5);

\node (An) at (0,-2.5) {$A_n=$};
\node (A0) at (3.5+5,-2.5) {$B_n^0$};
\node (A1) at (10.5+5,-2.5) {$B_n^1$};
\node (A2) at (17.5+5,-2.5) {$B_n^2$};
\node (A3) at (24.5+5,-2.5) {$B_n^3$};

\draw[-](5,-7)--(28+5,-7);
\draw[-](5,-12)--(28+5,-12);
\draw[-](5,-7)--(5,-12);
\draw[-](7+5,-7)--(7+5,-12);
\draw[-](14+5,-7)--(14+5,-12);
\draw[-](21+5,-7)--(21+5,-12);
\draw[-](28+5,-7)--(28+5,-12);

\node (Bn) at (0,-9.5) {$2K_{n-1}+A_{n-2}=$};
\node (B0) at (3.5+5,-9.5) {$A_{n-2}$};
\node (B1) at (10.5+5,-9.5) {$A_{n-1}$};
\node (B2) at (17.5+5,-9.5) {$A_{n-1}$};
\node (B3) at (24.5+5,-9.5) {$2CK_{n-1}$};

\node (f0) at (8,-6) {$f_0$};
\node (f1) at (15,-6) {$f_1$};
\node (f2) at (22,-6) {$f_2$};
\node (f3) at (29,-6) {$f_3$};

\foreach \from/\to in {A0/B0,A1/B1,A2/B2,A3/B3}
            \draw[->] (\from) -- (\to);

\end{tikzpicture}

We start by defining a bijection $f_0:B_n^0 \rightarrow A_{n-2}$
in the following way:
we first remove the last element of $\pi$ (note that after this step we obtain a non-king of order $n-1$) and then we remove the first element of the resulting permutation.

For example, let $\pi=[426153]\in B_6^0$. Then $f_0(\pi)=[2413]\in A_{4}$.

In order to show that $f_0$ is bijective, we present the inverse function. Explicitly, for $\sigma=[a-1,\dots,a] \in A_{n-2}$ (or $\sigma=[a,\dots,a-1] \in A_{n-2}$), ${f_0}^{-1}(\sigma)$ is obtained from $\sigma$ by adding the element $a$ sequentially at the two sides of $\sigma$, i.e. to the left of $a-1$ and to the right of $a$. 
For example, let $\sigma=[2413]$. Then we first add $3$ to the left of $\sigma$ to get $\sigma'=[32514]$ and then add again $3$ to the right of $\sigma'$ to get $[426153]\in B_n^0$.

Next, we construct a function $f_1:B_n^1 \rightarrow A_{n-1}$ by removing from $\pi=[\pi_1,\dots,\pi_n]$ the element $max\{\pi_1,\pi_n\}$.  
For example:  
Let $\pi=[5137246]$, then  $f_1(\pi)=[513624]$. 

In a similar way, define $f_2:B_n^2 \rightarrow A_{n-1}$ by removing the element $min\{\pi_1,\pi_n\}$. For instance, $f_2([5724136])=[624135]$. 

We show now the inverse of the function $f_1$. 
First, if $\sigma=[a,\dots,a-1] \in A_{n-1}$, then $f_1^{-1}(\sigma)$ is obtained by adding $a+1$ at the end of $\sigma$, i.e. $f_1^{-1}(\sigma)=[a,l\dots,a-1,a+1] \in B_n^1$ (note that $l \neq a+2$ since otherwise $\sigma \notin A_{n-1}$). If $\sigma=[a-1,\dots,a]$ then $f_1^{-1}(\sigma)$ is obtained by adding $a+1$ at the beginning of $\sigma$, i.e. $f_1^{-1}(\sigma)=[a+1,a-1,\dots, l,a] \in B_n^1$, ($l \neq a+2$). \\
Similarly, we show now the inverse of $f_2$.
 First, if $\sigma=[a,\dots,a-1] \in A_{n-1}$, then $f_2^{-1}(\sigma)$ is obtained by adding $a-1$ at the beginning of $\sigma$, i.e. $f_2^{-1}(\sigma)=[a-1,a+1\dots l,a] \in B_n^2$ (note that $\l \neq a-2$). If $\sigma=[a-1,\dots,a]$ then $f_2^{-1}$ is obtained by adding $a-1$ at the end of $\sigma$, i.e. $f_2^{-1}(\sigma)=[a,l\dots,a+1,a-1] \in B_n^2$ (since $l \neq a-2$). \\
Next, we construct a mapping $f_3:B_n^3 \rightarrow CK_{n-1}$ which is $2$ to $1$, i.e. each element of $CK_{n-1}$ will have exactly two preimages. For $\pi \in B_n^3$, the function $f_3$ removes from $\pi$ the maximum of its two extreme elements.

For example, let $\pi=[364152] \in B_6^3$. Then $f_3(\pi)=[53142]\in CK_5$. 
Note that we also have: $f_3([531426])=[53142]$. 
In order to see that the function $f_3$ is indeed $2$ to $1$, we show how to go back from an arbitrary element of $CK_{n-1}$ 
to its two preimages. Let $\sigma=[a,\dots,b] \in CK_{n-1}$, then we have that $a \neq b \pm 1$, so we define $\pi_1$ to be the permutation obtained by adding 
$a+1$ after $b$ and let $\pi_2$ be the permutation obtained by adding $b+1$ before $a$. It is easy to see then that $f_3(\pi_1)=f_3(\pi_2)=\sigma$. 

From this data we have that $$|A_n|=|B_n^0|+|B_n^1|+|B_n^2|+|B_n^3|=|A_{n-2}|+2|A_{n-1}|+2|CK_{n-1}|=|A_{n-2}|+2|K_{n-1}|.$$

\begin{remark}
Note that by extending the concept of {\em separator} which was developed by the authors in \cite{separator}  to {\it edge separator}, we can consider the above functions $f_0,f_1,f_2$ as omitting the vertical edge separators from $\pi$.
If $|\pi_1-\pi_{n-1}|=1$ then $\pi_n$ is a (vertical) {\it edge separator}. 
If $|\pi_n-\pi_2|=1$ then $\pi_1$ is a (vertical) {\it edge separator}.
\end{remark}
\end{proof}

Using the above theorem we can prove the following conclusion. 
\begin{theorem}
\label{asy}
The asymptotic value of $|CK_n|/|K_n|$ is equal to 1 
\end{theorem}
\begin{proof}
By Theorem \ref{recursion} we have:  
$|A_{n}|=2|K_{n-1}|+|A_{n-2}|.$\\ Thus, $|K_n|-|CK_n|=2|K_{n-1}|+|K_{n-2}|-|CK_{n-2}|$ and we obtain:  
\begin{equation}
\label{kn}
|CK_n|/|K_n|=1-2|K_{n-1}|/|K_n|-|K_{n-2}|/|K_n|+|CK_{n-2}|/|K_n|    
\end{equation}
According to Robbins \cite{Ro}, when $n$ tends to infinity, 
the probability of picking a king permutation from $S_n$ approaches $e^{-2}$. 
Thus, the asymptotic value of $|K_{n}|$ is $n!e^{-2}$ and thus the asymptotic value of $|K_{n-1}|/|K_n|$ is $\frac{1}{n}\rightarrow 0$ and the asymptotic value of $|K_{n-2}|/|K_n|$ is $\frac{1}{n^2}\rightarrow 0$. As a result,
\begin{equation}
\label{kkn}
  0<|CK_{n-2}|/|K_n|<|K_{n-2}|/|K_n|\rightarrow 0  
\end{equation}
 Using \ref{kn} and \ref{kkn} we obtain $|CK_n|/|K_n|\rightarrow 1$   
\end{proof}

\section{A direct calculation of the distribution of cylindrical kings}
\label{generating}
\subsection{The distribution of cyclic bonds}
In \cite{H2} the distribution of the number of bonds was calculated, where the author used it, inter alia, to  evaluate the number of king permutations, while considering them as permutations without bonds.
In this section we introduce the distribution of our new concept, the cyclic bonds, and use it to calculate the distribution of the cylindrical king permutations. 
In a recent paper of the authors \cite{separator} the concept of marked permutations is used. Here, we use again this concept in order to calculate the distribution of the cyclic bonds over $S_n$. We start with the following definition:

\begin{definition}
 A {\em marked permutation} is a permutation $\pi \in S_n$ such that each cyclic bond is either chosen or neglected.  
 If several consecutive cyclic bonds are chosen, then they form a {\em run}. An entry that is not chosen is considered to be a run of length $1$ (a  trivial run). 
 Note that a run of a permutation might be ascending or descending.
\end{definition}

Note that each marked permutation $\pi \in S_n$ can be presented as a concatenation of runs. For example, the permutation $[83124567]$ contributes the marked permutation $[8/3/12/4567]$ (which consists of the $3$ runs $3,12,45678$), the marked permutation $[8/3/12/4567/]$ (which consists of the $4$ runs $8,3,12,4567$), as well as the marked permutation $[8/3/1/2/4567/]$ (which consists of the $5$ runs $8,3,1,2,4567$) and many more. Moreover, a marked permutation $\pi \in S_n$ with $m$ runs can be uniquely characterized by the following data: a vector containing the lengths of the runs (which is a composition $\lambda$ of $n$) and their directions (increasing or decreasing),
a permutation $\sigma \in S_m$ which determines their locations and $r$- the number of places the last run occupies at the beginning of the permutation $\pi$. 
This idea will be best explained by an example. 

\begin{example}
Let $\pi=[2/45/6/1/987/3]$. We write $\pi$ as a triple consisting of a 'directed' composition of $m=5$ parts $\lambda$, a permutation $\sigma \in S_5$ and $r=1$.   First, write $\pi$ as a sequence of runs: $b_1=1,b_2=32,b_3=45,b_4=6,b_5=987$.
Each run contributes its length to the composition. Then for each part, we add the sign $\uparrow$ if the corresponding run is increasing, the sign $\downarrow$ if the run is decreasing and no arrow if the run is of length $1$. 
In our case we get $\lambda=(1,2\downarrow,2\uparrow,1,3\downarrow)$. Now, $\sigma \in S_5$ is the permutation induced by the order of the blocks. In our case, since the block $b_1=1$ is located as the third block of $\pi$, (note that the digit $2$ does not constitute a separate block but a part of the block $32$), $b_2=32$ is located as the fifth block, $b_3=45$ is placed as the first block, $b4=6$ is the second block, and $b_5=987$ is located as the forth block, we have: $\sigma=[34152]$. The marked permutation $\pi$ is now uniquely defined by the pair $(\lambda,\sigma,r)$. 
Note that if the marked permutation is $\pi=[45/6/1/987/32]$ then we have the same data except for the fact the now $r=0$.  
\end{example}

In order to calculate the distribution of the number of cyclic bonds, we count marked permutations and use the inclusion-exclusion principle to extract the desired distribution. 

The next theorem presents a generating function for the number of cyclic bonds. 

For each $k,n \geq 0$, we denote $$a_{n,k}=|\{\pi \in S_n\mid cbnd(\pi)=k\}|.$$  Furthermore, let $$H(z,u)=\sum\limits_{n\geq 1}\sum\limits_{k\geq 0}{a_{n,k}z^n u^k}$$ be the generating function of the number of cylic bonds. 
We have now: 
\begin{theorem}
\label{TM:Generating Function for cyclic bonds}

$$H(z,u)=-2z^2(u-1)+\sum_{m\geq 1}{m! z^{m-1}\left(\frac{1+z(u-1)}{1-z(u-1)}\right)^{m-1}\left(z+\frac{2z(2z(u-1)-(z(u-1))^2)}{(1-z(u-1))^2}\right)}.$$
\end{theorem}
\begin{proof}
We first calculate the
generating function of {\bf marked} cyclic bonds.  Let $$F(z,u)=\sum\limits_{n \geq 0}\sum\limits_{k \geq 0} f_{n,k}z^nu^k$$ where $f_{n,k}$ is the number of permutations of order $n$ with $k$ marked cyclic bonds.
We count the marked permutations by considering each run separately. Let $m$ be the number of runs in a permutation. There are $m!$ ways 
 to arrange the runs in each marked permutation. 
For each such an arrangement, we distinguish between the last run of a permutation and the other runs. (If there is only one run then it will be considered the last one).  We start with the first $m-1$ runs. 
   Each run of length one contributes $z$, while each run of length $k \geq 2$ has $k-1$ regular bonds and can be increasing or decreasing, so it contributes $2z^k u^{k-1}$. This gives us $(z+2z^2u+2z^3u^2+2z^4u^3+\ldots)^{m-1}$.
   
  Now we discuss the last run that starts at the end of the permutation and may emerge at the beginning of the permutation (like the run $1234$ in the permutation $[{\bf4}/67/5/ {\bf123}]$). If it is of length $1$ then it contributes $z$. If it has length
 $k\geq 2$ then it has $k-1$ cyclic bonds and
may be decreasing or increasing. Moreover, its $k-r$ first digits are placed at the end of the permutation, and the other $r$ digits at the beginning of the permutation for $0 \leq r \leq k-1$. Thus its contribution is $2kz^k u^{k-1}$. This gives us $(z+2\cdot 2z^2u+2\cdot 3z^3u^2+2\cdot 4z^4u^3+\ldots)$. Using this technique, we actually count the permutations [12] and [21] twice (while taking $r=0$ and $r=1$ for the last block in the case that $m=1$) so we have to subtract $2z^2u$. 
 This leads to, $$
F(z,u)=-2z^2u+\sum_{m \geq 1}{m!(z+2z^2u+2z^3u^2+2z^4u^3+\ldots)^{m-1}(z+4z^2u+6z^3u^2+8z^4u^3+\ldots)}$$\\
$${=-2z^2u+\sum_{m\geq 1}{m! z^{m-1}\big(\frac{1+zu}{1-zu}\big)^{m-1}}}(z+\frac{2z(2zu-(zu)^2)}{(1-zu)^2}).$$ 
  Now, we can use $F(z,u)$ to obtain $H(z, u)$. 
  Since $F(z,u)$ counts the \textbf{marked} cyclic bonds and $H(z, u)$ counts \textbf{every} cyclic bond, using the inclusion–exclusion principle it follows that  $H(z,u)=F(z,u-1)$, and thus:  
$$H(z,u)=-2z^2(u-1)+\sum_{m\geq 1}{m! z^{m-1}\left(\frac{1+z(u-1)}{1-z(u-1)}\right)^{m-1}\left(z+\frac{2z(2z(u-1)-(z(u-1))^2)}{(1-z(u-1))^2}\right)}.$$

\end{proof}

\subsection{The distribution of cylindrical king permutations}

Substituting $u=0$ in Theorem \ref{TM:Generating Function for cyclic bonds} we can easily find a generating function for cylindrical king permutations, i.e., a function of the form $CK(z)=\sum\limits_{n\geq 1}{|CK_n|z^n}$.

\begin{theorem}
\label{generating of the cyclic kings}
$$CK(z)=2z^2+\sum_{m \geq 1}{m!z^{m-1}\left(\frac{1-z}{1+z}\right)^{m-1}}\left(\frac{z(1-2z-z^2)^2}{(1+z)^2}\right).$$ 
\end{theorem}\qed

\section{Acknowledgments}
The authors want to thank Yuval Roichman for fruitful discussions.


\begin{thebibliography}{}
\bibitem{AM}{M. Abramson \& W. O. J. Moser}, {Permutations without rising or falling w-sequences}, The Annals of Mathematical Statistics, Volume 38, Number 4 (1967), 1245-1254.



\bibitem{BHT}{D. Bevan, C. Homberger and B. E. Tenner},{ Prolific permutations and permuted packings: downsets containing many large patterns},Journal of Combinatorial Theory, Series A, Volume 153, January 2018, Pages 98-121.

\bibitem{separator}{E. Bagno, E. Eisenberg, S. Reches, M. Sigron},{Separators - A New Statistic for Permutations}, https://arxiv.org/pdf/1905.12364.pdf, 2019.

\bibitem{BERS_KINGS}{E. Bagno, E. Eisenberg, S. Reches, M. Sigron}, {On the poset of King-Non-Attacking permutations}, arXiv:1905.02387.


\bibitem{PR}{P. Adeline, R. Dominique},{ Simple permutations poset},  https://arxiv.org/pdf/1201.3119.pdf, 2012.



\bibitem{Cellini}{P. Cellini}, {Cyclic Eulerian Elements}, Europ. J. Combinatorics , 545–552, 1998.

\bibitem{H2}
{C. Homberger}, {\it Counting Fixed-Length Permutation Patterns}. Online Journal of Analytic Combinatorics, Issue  7, (2012). 

\bibitem{Petersen}{T. K Petersen}, {Cyclic Descents and P-Partitions}, Journal of Algebraic Combinatorics, 22, 343–375, 2005.

\bibitem{Ro}{D. P. Robins},{ The probability that neighbors remain neighbors after random rearrangements},  The American Mathematical Montly, Vol. 87, 1980 - issue 2. pp 122-124. 

\bibitem{Vatter}{V.\ Vatter},{ Permutation classes},
in: Handbook of Combinatorial Enumeration (M.\ Bona,
ed.), CRC Press 2015, pp.\ 753--833.

\bibitem{oeisHer}{Sloane, N. J. A.}, Sequence A002464 in "The On-Line Encyclopedia of Integer Sequences."

\bibitem{JELIS}{George Jelliss}, {Cylinder Chess}, Variant Chess, Volume 3, Issue 22, pages 32–33, 1996.
\bibitem{Sl}{Sloane, N. J. A.}, Sequence A002464 in "The On-Line Encyclopedia of Integer Sequences."
\end{thebibliography}
\end{document}